\documentclass{article}[12pt]
\usepackage{a4}
\usepackage[english]{babel}
\usepackage[T1]{fontenc}
\usepackage[all]{xy}
\usepackage{amssymb}
\usepackage{amsmath}
\usepackage{amsthm}
\usepackage{multirow}
\usepackage{graphicx}
\usepackage{upgreek}

\newtheorem{thm}{Theorem}[section]
\newtheorem{prop}[thm]{Proposition}
\newtheorem{lem}[thm]{Lemma}

\newtheorem{defn}[thm]{Definition}
\newtheorem{expl}[thm]{Example}
\newtheorem{rmk}[thm]{Remark}

\numberwithin{equation}{section}

\newcommand{\lra}{\longrightarrow}

\newcommand{\mcalo}{\mathcal{O}}

\newcommand{\mcalc}{\mathcal{C}}

\newcommand{\om}{\mathcal{O}(M)}

\newcommand{\okm}{\mathcal{O}_k(M)}

\newcommand{\mcalf}{\mathcal{F}}
\newcommand{\mcald}{\mathcal{D}}

\newcommand{\mcalm}{\mathcal{M}}
\newcommand{\mcalp}{\mathcal{P}}
\newcommand{\mcalb}{\mathcal{B}}

\newcommand{\mcall}{\mathcal{L}}
\newcommand{\mcalu}{\mathcal{U}}

\newcommand{\hra}{\hookrightarrow}
\newcommand{\lla}{\longleftarrow}

\newcommand{\ra}{\rightarrow}
\newcommand{\Bt}{\widetilde{\mcalb}}

\newcommand{\mcalh}{\mathcal{H}}

\newcommand{\mcalw}{\mathcal{W}}

\newcommand{\mcals}{\mathcal{S}}

\newcommand{\Sp}{\widetilde{\text{sp}}}

\newcommand{\spim}{\widetilde{\text{sp}}^{\infty}(M)}
\newcommand{\mcalbb}{\overline{\mcalb}}
\newcommand{\spkm}{\widetilde{\text{sp}}^k(M)}
\newcommand{\mcalpa}{\mcalp_{k\abb}}
\newcommand{\mcalfa}{\mcalf_{\abb}}
\newcommand{\mcalla}{\mcall_{\abb}}
\newcommand{\mcalfb}{\overline{\mcalf}}
\newcommand{\mcalfab}{\overline{\mcalf}_{\abb}}
\newcommand{\stk}{\sim_k}
\newcommand{\Bb}{\overline{\mcalb}}
\newcommand{\Bpb}{\overline{\mcalb'}}
\newcommand{\Bbspkm}{\mcalbb(\spkm)}
\newcommand{\Lb}{\overline{L}}
\newcommand{\hfib}{\text{hofiber}}

\newcommand{\abb}{\mathbb{A}}
\newcommand{\Ah}{\widehat{\mathbb{A}}}

\title{ \textbf{Classification of polynomial functors in manifold calculus, Part I}}
\date{}

\author{Paul Arnaud Songhafouo Tsopm\'en\'e \\
Donald Stanley}

\begin{document}
\maketitle

\begin{abstract}  
For a smooth manifold $M$, we define a topological space $\spkm$, and show that  polynomial functors $\om \to \mcalm$ of degree $\leq k$ from the poset of open subsets of $M$ to a simplicial model category can be classified by a version of linear functors from $\mcalo(\spkm)$ to $\mcalm$. 
\end{abstract}

\tableofcontents


\section{Introduction}

\sloppy

Let $M$ be a smooth manifold, and let $\om$ be the poset of open subsets of $M$ ordered by inclusion.  Manifold calulus, as defined in \cite{wei99}, is the study of good contravariant functors from $\om$ to spaces. A good functor turns isotopy equivalences into weak equivalences. Being a calculus of functors, the idea of manifold calculus is to try to approximate a given functor by functors that are \textit{polynomial} (see Definition~\ref{defn:PF}). 
This paper is concerned with the study of polynomial functors. 

Such functors were studied by some people including  Weiss \cite{wei99}, Pryor \cite{pryor15} and the authors \cite{paul_don17-2}.  In \cite[Theorem 6.1]{wei99}, Weiss characterizes polynomial functors of degree $\leq k$. Specifically, he shows that these are determined by their values on $\mcalo_k(M)$, the full subposet of $\om$ whose objects are open subsets diffeomorphic to the disjoint union of at most $k$ balls.  In the same spirit as Weiss,  Pryor \cite[Theorem 6.12]{pryor15} generalizes this characterization of a polynomial functor $F \colon \om \to \text{Top}$ of degree $\leq k$ by its restriction to a full  subposet $\mcalb_k(M) \subseteq \okm$ defined as follows. Let $\mcalb$ be a basis for the topology of $M$. Each element of $\mcalb$ is required to be diffeomorphic to an open ball. Such a $\mcalb$ is called \textit{good basis}.  The objects of $\mcalb_k(M)$ are  exactly those which are a union of at most $k$ pairwise disjoint elements of $\mcalb$. In \cite[Theorem 1.1]{paul_don17-2}, the authors show that Pryor's result holds for a larger class of functors, namely those from $\om$ to an arbitrary simplicial model category $\mcalm$. 


In this paper we classify polynomial functors $F \colon \om \to \mcalm$ of degree $\leq k$ by a version of linear functors  from the poset $\mcalo(\spkm)$ of open subsets of a topological space $\spkm$ to $\mcalm$. The space $\spkm$ is a variation of the $k$-fold symmetric power $\text{sp}^k(M):= M^k \slash \Sigma_k$. Specifically,  $\spkm$  is defined to be the quotient of $M^k$ by some equivalence relation $\sim_k$  (see Definition~\ref{defn:spkm}).   
We use $\spkm$ (instead of $\text{sp}^k(M)$) because it has better properties for our purposes. The first is that for every $k$, there is a canonical injection $\spkm \to \Sp^{k+1}(M)$ which allows us to define the space $\spim$ as the colimit of $\Sp^0(M) \to \Sp^1(M) \to \cdots$.  This latter space is useful when dealing with analytic functors. The second property is that $\spkm \backslash \Sp^{k-1}(M)$ is equal to $F_k(M)$, the unordered configuration space of $k$ points in $M$, which is used in \cite{paul_don18, wei99} to classify homogeneous functors.  

The notion of linear functor $F \colon \mcalo(\spkm) \to \mcalm$ we consider in this paper can be briefly defined as follows. 
First,  we define an important full suposet $\Bb(\spkm) \subseteq \mcalo(\spkm)$ out of a good basis $\mcalb$. Roughly speaking, an object of $\Bb(\spkm)$ is a quotient $B_1 \times \cdots \times B_k \slash \sim_k$ where each $B_i$ belongs to $\mcalb$, and for every pair $\{i, j\}$ either $B_i = B_j$ or $B_i \cap B_j = \emptyset$ (see Definition~\ref{defn:mcalbb}). A functor $F \colon \mcalo(\spkm) \to \mcalm$ is called linear if it is determined by its values on $\Bb(\spkm)$ (see Definition~\ref{defn:blinear}). As we will see below, the poset $\Bb(\spkm)$ (together with $\mcalb_k(M)$) is useful to relate polynomial functors $\om \to \mcalm$ of degree $\leq k$ to linear functors $\mcalo(\spkm) \to \mcalm$. There is also a notion of goodness for functors out of $\mcalo(\spkm)$ (see Definition~\ref{defn:bgood}).

We now proceed with the statement of our results. Let  $\mcall^{\mcalb} (\mcalo(\spkm); \mcalm)$ denote the category of good linear functors $\mcalo(\spkm) \to \mcalm$ (see  Definition~\ref{defn:mcallb}). The superscript $\mcalb$ is meant to indicate that this category depends on $\mcalb$. Actually, this is not the case as  we prove Theorem~\ref{thm:lb_lbp} which states that given any other good basis $\mcalb'$, the category $\mcall^{\mcalb} (\mcalo(\spkm); \mcalm)$ is equal to $\mcall^{\mcalb'} (\mcalo(\spkm); \mcalm)$. 
Thanks to this, one can drop the superscript $\mcalb$. 
For our purposes in the follow-up paper, we also consider a full subcategory $\mcall_{\abb} (\mcalo(\spkm); \mcalm)$ of  $\mcall (\mcalo(\spkm); \mcalm)$, where $\abb = \{A_0, \cdots, A_k\}$ is a sequence of objects of $\mcalm$. The objects of $\mcall_{\abb} (\mcalo(\spkm); \mcalm)$ are functors $F$ such that for every $V \in \Bb(\spkm)$, $F(V) \simeq A_i$ whenever the \textit{degree} of $V$ is equal to $i$. The degree of $V = B_1 \times \cdots \times B_k \slash \sim_k$ is defined to be the number of components of $B_1 \cup \cdots \cup B_k$ (see Definition~\ref{defn:degree}). Now let $\mcalp_k(\om; \mcalm)$ be the category of  (good) polynomial functors $F \colon \om \to \mcalm$ of degree $\leq k$ (see Definition~\ref{defn:mcalpa}), and let $\mcalp_{k\abb} (\om; \mcalm)$ be the full subcategory of  $\mcalp_k(\om; \mcalm)$ whose objects are functors $F$ having the property that $F(U) \simeq A_i$ whenever $U$ is diffeomorphic to the disjoint union of exactly $i$ open balls (not necessarily in $\mcalb$). Define a functor 
\[
\Theta \colon \mcalp_k(\om; \mcalm) \to \mcall(\mcalo(\spkm); \mcalm)
\]
as $\Theta (F):= (\Phi(F))^!$, where $\Phi(F)$ is a functor from $\Bb(\spkm)$ to $\mcalm$ defined by $\Phi(F)(V) := F \lambda (V)$. Here $(-)^!$ stands for the homotopy right Kan extension functor, while $\lambda \colon \Bb(\spkm) \to \mcalb_k(M)$ is the projection functor $\lambda (V):= B_1 \cup \cdots \cup B_k$ (see Definition~\ref{defn:lamb_theta}).  
By definition, $\Theta(F) \in \mcalla(\mcalo(\spkm); \mcalm)$ whenever $F$ belongs to $\mcalpa(\om; \mcalm)$. This defines a functor 
\[
\Theta_{\abb} \colon \mcalp_{k\abb}(\om; \mcalm) \to \mcall_{\abb}(\mcalo(\spkm); \mcalm), \quad F \mapsto \Theta (F). 
\]
 
The following is the main result of this paper. 



\begin{thm}  \label{thm:main}
	Let $M$ be a smooth manifold, and let $\mcalm$ be a simplicial model category. Let $\abb = \{A_0, \cdots, A_k\}$ be a sequence of objects of $\mcalm$. Then the functors $\Theta$ and $\Theta_{\abb}$ above are both weak equivalences (in the sense of Definition~\ref{defn:we}). 
\end{thm} 
Using the terminology of Definition~\ref{defn:we}, Theorem~\ref{thm:main} implies that the category of polynomial functors $\om \to \mcalm$ of degree $\leq k$ is weakly equivalent to the category of linear functors $\mcalo(\spkm) \to \mcalm$. This is interesting because it is easier to work with linear functors.

The strategy of the proof of Theorem~\ref{thm:main} is to construct a functor 
\[
\Lambda \colon \mcall(\mcalo(\spkm); \mcalm) \to \mcalp_k(\om; \mcalm),
\]
in the other direction and show that it has the required properties. 
The definition of $\Lambda$ is simple:  $\Lambda(F) := (\Psi(F))^!$, where $\Psi(F) \colon  \mcalb_k(M) \to \mcalm$ is given by $\Psi(F) := F \theta$. Here $\theta$ is a functor from $\mcalb_k(M)$ to $\Bb(\spkm)$ defined as $\theta(U) := U^k \slash \sim_k$. 

We end this introduction with a couple of remarks, a fact about homogeneous functors, and what we plan to do in the follow-up paper. 

\begin{rmk}
	Theorem~\ref{thm:main} also holds for analytic functors, that is, for functors $F \colon \om \to \mcalm$ that are determined by their values on $\mcalb_{\infty}(M) := \cup_{k \geq 0} \mcalb_k(M)$. In fact, using the same strategy as above, one can show that the category of such functors is weakly equivalent to the category of linear functors $\mcalo(\spim) \to \mcalm$. As before, an object of this latter category is determined by a functor $\Bb(\spim) \to \mcalm$, where $\Bb(\spim):= \cup_{k \geq 0} \Bb(\spkm)$. 
\end{rmk}	
 
 \begin{rmk}
 	Associated with a functor $F \colon \om \to \mcalm$ is its Taylor tower $\{T_rF\}_{r \geq 0}$, where $T_rF$ is the $r$th polynomial approximation to $F$. By definition, $T_rF$ is the homotopy right Kan extension of $F \colon \mcalb_r(M) \to  \mcalm$ along the inclusion $\mcalb_r(M) \hra \om$. So $T_rF$ is determined by the restriction $F|\mcalb_r(M)$. The analog of this for functors out of $\mcalo(\spim)$ is what we expect: the Taylor tower of $F \colon \mcalo(\spim) \to \mcalm$ is given by the restrictions $F |\Bb(\spkm)$.  
 \end{rmk} 


Given $F \colon \om \to \mcalm$, the difference between $T_rF$ and $T_{r-1}F$, or more precisely the homotopy fiber of the canonical map $T_rF \to T_{r-1}F$, is referred to as the $r$th homogeneous layer of $F$. Section~\ref{section:hl}, which is the last section of this paper, defines  the analog of homogeneous layer for functors out of $\mcalo(\spkm)$. Specifically, given $F \in \mcall(\mcalo(\spkm); \mcalm)$ we define a new functor $\Lb_r(F) \in \mcall(\mcalo(\spkm); \mcalm), r \leq k,$ and prove Theorem~\ref{thm:hl} which roughly says that $\Lambda(\Lb_r(F)) \simeq L_r(\Lambda (F))$.  (Note that our definition of $\Lb_r(F)$ does not involve $\Theta$ or $\Lambda$ mentioned above.)


In work in progress, the authors are trying to construct a filtered space $\widehat{\abb}$, out of $\abb = \{A_0, \cdots, A_k\}$, that classifies the objects of $\mcalp_{k\abb}(\om; \mcalm)$. Specifically, we plan to show that weak equivalence classes of such objects are in one-to-one correspondence with filtered homotopy classes  of maps  $\spkm \to \widehat{\abb}$. The first step of this is to construct an equivalence between $\mcalp_{k\abb}(\om; \mcalm)$ and $\mcall_{\abb}(\mcalo(\spkm); \mcalm)$, which is nothing but Theorem~\ref{thm:main} above. The next step, which is the purpose of the second part of this paper,  is to relate linear functors on $\mcalo(\spkm)$ to maps $\spkm \to \Ah$ using the basic outline of $\cite{paul_don18}$. The piecewise linear topology involved is more complicated.

\textbf{Outline.} 
In Section~\ref{sec:spkm}, we  define the space $\spkm$. Then, in Section~\ref{sec:mcalp_mcalfb}, we define the  poset $\Bb(\spkm)$ and the functors $\lambda \colon \Bb(\spkm) \to \mcalb_k(M)$ and  $\theta \colon \mcalb_k(M) \to \Bb(\spkm)$ which allow us transport functors $\mcalb_k(M) \to \mcalm$ to functors $\mcalo(\spkm) \to \mcalm$ and vice versa. An important category $\mcalfb(\Bb(\spkm); \mcalm)$  of functors out of $\Bb(\spkm)$ is introduced. We also define the category $\mcalp_k(\om; \mcalm)$ of polynomial functors of degree $\leq k$, and show that it is weakly equivalent to $\mcalfb(\Bb(\spkm); \mcalm)$. Section~\ref{sec:linear_func} defines, out of a good basis $\mcalb$, the category $\mcall^{\mcalb}(\mcalo(\spkm); \mcalm)$ of linear functors from $\mcalo(\spkm)$. The main result here, Theorem~\ref{thm:lb_lbp}, says that this latter category does not depend on the choice of $\mcalb$. The proof of this  occupies the rest of the section. 
In Section~\ref{sec:main_thm}, we show that the category $\mcalfb(\Bb(\spkm); \mcalm)$ is weakly equivalent to $\mcall(\mcalo(\spkm); \mcalm)$, which completes the proof of Theorem~\ref{thm:main}. Section~\ref{section:hl} deals with homogeneous layers as we just mentioned.

\textbf{Notation, convention, etc.} Throughout this paper the letter $M$ stands for a smooth manifold, while $\mcalm$ stands for a simplicial model category. A good basis $\mcalb$ for the topology of $M$ is fixed (see Definition~\ref{defn:gb_mcalbk}).  For $A_0, \cdots, A_k \in \mcalm$, we let $\abb$ denote the set $\{A_0, \cdots, A_k\}$. The term \lq\lq cofunctor\rq\rq{} means contravariant functor, while the term  \lq\lq functor\rq\rq{} is used for  covariant functors.  As usual, weak equivalences in the category of (co)functors into $\mcalm$ are natural transformations which are objectwise weak equivalences.

\textbf{Acknowledgment.}  This work has been supported  the Natural Sciences and Engineering Research Council of Canada (NSERC), that the authors acknowledge.

\section{The space $\spkm$}   \label{sec:spkm}

This section is dedicated to the definition of the space $\spkm$ that appears in the introduction. Roughly speaking, $\spkm$ is a variation of the  $k$-fold symmetric power $M^k \slash \Sigma_k$ (see Definition~\ref{defn:spkm}). The collection $\{\spkm\}_{k \geq 0}$ comes equipped with canonical injections $\Sp^{k-1}(M) \to \spkm$, and this enables us to define  $\Sp^{\infty}(M)$ as the colimit of the obvious diagram. 


Let $\Delta$ be the standard simplicial category. Recall that objects of $\Delta$ are $[k] = \{0, \cdots, k\}, k \geq 0,$ and morphisms are non-decreasing maps. 
Define a simplicial object $X_{\bullet}^M \colon \Delta \to \text{Top}$ in topological spaces as follows. 
\[
X^M_k := M^{k+1} = \underbrace{M \times \cdots \times M}_{k+1}.
\]
The set $X^M_k$ is endowed with the product topology. The face and degeneracy maps are defined by forgetting and repeating points.  Specifically, for $0 \leq i \leq k$, define $d_i \colon X_k^M \to X_{k-1}^M$ as 
\[
d_i(x_1, \cdots, x_{k+1}) := (x_1, \cdots, \widehat{x_{i+1}}, \cdots, x_{k+1}),
\] 
where  ``$\widehat{x_{i+1}}$'' means taking out $x_{i+1}$.  For $0 \leq j \leq k$, define $s_j \colon X_k^M \to X_{k+1}^M$ as 
\[
s_j(x_1, \cdots, x_{k+1}) := (x_1, \cdots, x_j, x_{j+1}, x_{j+1}, x_{j+2}, \cdots, x_{k+1}). 
\]
From the definition, one can easily see that a non-degenerate $k$-simplex of $X_{\bullet}^M$ is an ordered configuration of $k+1$ points in $M$. If $f$ is a morphism of $\Delta$, we write $f^*$ for $X_{\bullet}^M(f)$. 

\begin{defn} \label{defn:sim_sp}
	For $k \geq 1$, define $\stk$ as the equivalence relation on $M^k = X_{k-1}^M$ generated by $x =(x_1, \cdots, x_k)$ $\stk$ $(y_1, \cdots, y_k) =y$ if and only if
	\begin{enumerate}
		\item[$\bullet$] either there is a permutation $\sigma \in \Sigma_k$ on $k$ letters such that $y_i = x_{\sigma_i}$ for all $i$,
		\item[$\bullet$] or there is a non-degenerate simplex $z \in X_{l-1}^M = M^l$, for some $l \leq k$, and non-decreasing surjections $f_x, f_y \colon [k-1] \to [l-1]$ such that $f_x^*(z) = x$ and $f_y^*(z) = y$. 
	\end{enumerate}  
\end{defn}

\begin{expl}
	Let $k =3$, and let $x_1, x_2 \in M$.
	\begin{enumerate}
		\item[(i)] One has $(x_1, x_1, x_2) \  \stk \  (x_1, x_2, x_1) \  \stk \ (x_2, x_1, x_1)$.
		\item[(ii)] Suppose $x_1 \neq x_2$. One has $(x_1, x_1, x_2) \  \stk \ (x_1, x_2, x_2)$ since $z = (x_1, x_2) \in X_1^M = M^2$ is non-degenerate, and $s_0(z) = (x_1, x_1, x_2)$ and $s_1(z) = (x_1, x_2, x_2)$.
		\item[(iii)] Combining (i) and (ii), one has for example $(x_1, x_2, x_1) \  \stk \ (x_1, x_2, x_2)$. 
	\end{enumerate}  
\end{expl}

\begin{defn}  \label{defn:spkm}
	Consider the equivalence relation $\stk$ from Definition~\ref{defn:sim_sp}.  For $k \geq 0$, define 
	\[
	\spkm := \left\{ \begin{array}{ccc}
	\emptyset  &  \text{if}  & k = 0  \\
	M^k \slash \stk  &  \text{if}  & k \geq 1.
	\end{array}  \right.
	\]
\end{defn}

We endow $\spkm, k \geq 1,$ with the quotient topology. 

\begin{rmk} 
	From the definition, one has $\Sp^1(M) = M$, $\Sp^2(M) = \text{sp}^2(M)$, where $\text{sp}^k(M):= M^k \slash \Sigma_k$ is the standard $k$-fold symmetric power of $M$. For higher $k$'s, $\spkm$ is different from  $\text{sp}^k(M)$. In fact, $\spkm$ can be viewed as some quotient of $\text{sp}^k(M)$.
\end{rmk} 

There are maps $\varphi_k \colon \Sp^{k-1}(M) \to \spkm$ defined as 
\[
\varphi_k[(x_1, x_2,  \cdots, x_{k-1})] := [(x_1, x_1, x_2, \cdots, x_{k-1})].
\]
The space $\spkm$ is related to the  unordered configuration space $F_k(M)$ of $k$ points in $M$ as follows. 

\begin{prop}
	For $k \geq 1$, the map $\varphi_k \colon \Sp^{k-1}(M) \to \spkm$ is injective. Moreover, if we take $\varphi_k(\Sp^{k-1}(M)) \cong \Sp^{k-1}(M)$ away from $\spkm$, we left with  $F_k(M)$. That is,
	\[
	\spkm \backslash \Sp^{k-1}(M) = F_k(M). 
	\]
\end{prop}

\begin{proof}
	This follows immediately from the definitions. 
\end{proof}

\begin{defn}
	Define $\spim$ as the colimit of $\Sp^0(M) \stackrel{\varphi_1}{\to} \Sp^1(M) \stackrel{\varphi_2}{\to} \cdots$. That is,
	\[
	\spim:= \text{colim} \left(\Sp^0(M) \stackrel{\varphi_1}{\to} \Sp^1(M) \to \cdots \to \Sp^{k-1}(M) \stackrel{\varphi_k}{\to} \spkm \to \cdots \right)
	\]
\end{defn}


\section{Relating polynomial functors to functors out of $\Bb(\spkm)$}   \label{sec:mcalp_mcalfb}


The goal of this section is to define the poset $\Bb(\spkm)$ and prove Proposition~\ref{prop:mcalpk_mcalfb}, which roughly says that the category of polynomial cofunctors $\om \to \mcalm$ of degree $\leq k$ is weakly equivalent, in the sense of Definition~\ref{defn:we}, to  a certain category of cofunctors $\mcalbb(\spkm) \to \mcalm$.


\subsection{The poset $\Bb(\spkm)$}  \label{subsection:Bbspkm}

We begin with the following.

\begin{defn}  \label{defn:gb_mcalbk}
	\begin{enumerate}
		\item[(i)]  A basis $\mcalb$ for the  topology of $M$ is called \emph{good} if each of its elements is diffeomorphic to an open ball. 
		\item[(ii)] For a good basis $\mcalb$, define $\mcalb_k(M) \subseteq \om$ as the full subposet whose objects are unions of at most $k$ pairwise disjoint elements of $\mcalb$. 
	\end{enumerate}
\end{defn}

Note that the empty set is an object of $\mcalb_k(M)$. Let $\mcalo(\spkm)$ be the poset of open subsets of $\spkm$ ordered by inclusion, and let  $\stk$ be the equivalence relation from Definition~\ref{defn:sim_sp}.

\begin{defn} \label{defn:mcalbb}
	Let $\mcalb$ be a good basis for the topology of $M$. Define $\Bb(\spkm) \subseteq \mcalo (\spkm)$ as the full subposet  whose objects are either of the form
	\begin{equation} \label{eqn:type1}
	\left( \coprod_{f: \xymatrix{\{1, \cdots, k\} \ar@{->>}[r] & \{1, \cdots, r\}}} B_{f(1)} \times \cdots B_{f(k)} \right) \slash \stk,
	\end{equation}
	or of the form 
	\begin{equation} \label{eqn:type2}
	\left( \coprod_{g: \xymatrix{\{1, \cdots, k\} \ar[r] & \{1, \cdots, r\}}} B_{g(1)} \times \cdots B_{g(k)} \right) \slash \stk,
	\end{equation}
	where $B_1, \cdots, B_r$ are pairwise disjoint elements of $\mcalb$. Here $f$ runs  over the set of surjective maps from $\{1, \cdots, k\}$ to $\{1, \cdots, r\}$, while $g$ runs over the set of all maps $\{1, \cdots, k\} \to \{1, \cdots, r\}$. The poset $\Bb(\spkm)$ is also required to contain the empty set as one of its objects. 
\end{defn} 

We say that (\ref{eqn:type1}) and (\ref{eqn:type2}) are \textit{generated} by $B_1, \cdots, B_r$. For example, in Figure~\ref{fig:mcalbb} below, $U$ has only one generator, namely $B_1$, while $V$ (respectively $W$) is of the form (\ref{eqn:type1}) (respectively (\ref{eqn:type2})) and is generated by $B_1$ and $B_2$. 

\begin{figure}[!ht]
	\centering
	\includegraphics[scale = 0.7]{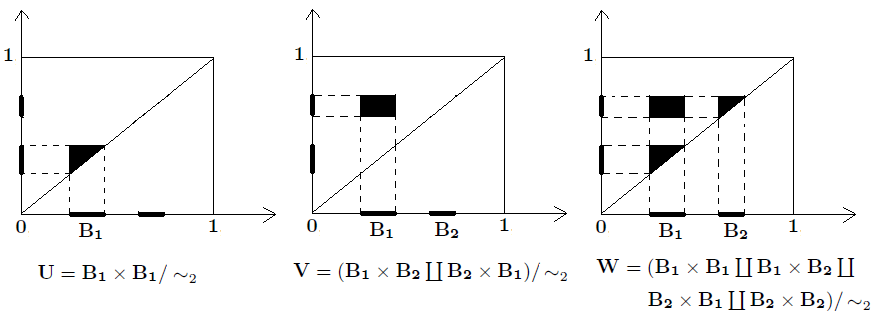}
	\caption{Example of objects of $\Bb(\Sp^2(M))$ with $M= [0, 1]$} \label{fig:mcalbb}
\end{figure}

\begin{defn}  \label{defn:lamb_theta}
	Define two functors $\lambda \colon \Bb(\spkm) \to \mcalb_k(M)$ and $\theta \colon \mcalb_k(M) \to \Bb(\spkm)$ as  
	\begin{enumerate}
		\item[(i)] $\lambda(\emptyset) = \emptyset$,  $\lambda(V) = B_1 \cup \cdots \cup B_r$ if $V$ has one of the forms (\ref{eqn:type1}) or (\ref{eqn:type2}), and 
		\item[(ii)] $\theta(U) = U^k \slash \stk$. 
	\end{enumerate} 
\end{defn} 

For example, in Figure~\ref{fig:mcalbb}, $\lambda(U) = B_1, \theta(B_1) = U$, $\lambda(V) = B_1 \cup B_2 = \lambda(W)$, and $\theta (B_1 \cup B_2) = W$. 

\begin{rmk} \label{rmk:lt=id}
	By the definitions, one has $\lambda \theta = id$. One also has an inclusion $\tau_V \colon V \hra \theta \lambda (V)$ for all $V$. In fact $V$ is one of the connected components of $\theta \lambda (V)$. 
\end{rmk} 

\begin{rmk} \label{rmk:objects_Bb}
	For $V \in \Bb(\spkm)$ generated by $B_1, \cdots, B_r$, define $\widehat{V}$ as the object of $\Bb(\spkm)$ of the form (\ref{eqn:type1}) which has the same generators as $V$. So  $V$ has the form (\ref{eqn:type1}) if and only if $\widehat{V} = V$.   The point of the remark is that every $V \in \Bb(\spkm)$ of the form (\ref{eqn:type2}) can be written as $V = \theta \lambda (\widehat{V})$. For example, in Figure~\ref{fig:mcalbb}, one has $\widehat{W} = V$ and $W= \theta \lambda (\widehat{W})$. The construction $\widehat{(-)}$ will be also used in Section~\ref{sec:linear_func} to show that certain categories are contractible, and in  Section~\ref{section:hl} to define the homogeneous layer of cofunctors out of $\mcalo(\spkm)$.  
\end{rmk}

The following definition is that of the concept of degree of an object of $\Bb(\spkm)$, which is used in different places in this paper. 

\begin{defn}  \label{defn:degree}
	 The \emph{degree} of $V \in \Bb(\spkm)$, denoted $\deg(V)$, is the number of generators of $V$.
\end{defn}	

For instance, in Figure~\ref{fig:mcalbb}, $\deg(U) = 1$ and $\deg(V) = 2 = \deg(W)$. Any object of the form (\ref{eqn:type1}) or (\ref{eqn:type2}) is of degree $r$.  Of course, the degree of the emptyset is $0$. 


\subsection{Relating polynomial functors to functors out of $\Bb(\spkm)$}

We now define an important category of cofunctors out of $\Bb(\spkm)$ and relate this to the catogory of polynomial cofunctors $\om \to \mcalm$ of degree $\leq k$. We first need to define a class of weak equivalences in $\Bb(\spkm)$.

\begin{defn} \label{defn:mcalw_Bb}
	Define $\mcalw_{\Bb(\spkm)} \subseteq \Bb(\spkm)$ as the subcategory with the same objects as $\Bb(\spkm)$. An inclusion $f \colon V \hra V'$ is declared  to be a morphism of $\mcalw_{\Bb(\spkm)}$ if it satisfies one of the following conditions:
	\begin{enumerate}
		\item[(a)] $\deg(V) = \deg(V')$ and $\pi_0(f) \colon \pi_0(V) \to \pi_0(V')$ is a bijection.
		\item[(b)] $f$ factors as $V \stackrel{g}{\to} W \stackrel{h}{\to} V'$ where $g$ and $h$ are both morphisms of $\Bb(\spkm)$ satisfying the following three conditions: $\pi_0(g)$ is a bijection, $\deg(V) = \deg(W)$, and $V' = \theta \lambda(W)$.  
	\end{enumerate}
A morphism of $\mcalw_{\Bb(\spkm)}$  is called \emph{weak equivalence}. 
\end{defn} 

For example, for every $V \in \Bb(\spkm)$, the inclusion $\tau_V \colon V \hra \theta \lambda(V)$ is a weak equivalence as it satisfies (b).  Note that morphisms of $\Bb(\spkm)$ satisfying  (a) are  honest weak equivalences as they are  homotopy equivalences in the traditional sense.  For our purposes, we also consider those satisfying (b) as weak equivalences.   

\begin{defn} \label{defn:mcalfb_spk}
	Let $\mcalbb(\spkm)$ be the poset from Definition~\ref{defn:mcalbb}. Define $\mcalfb(\mcalbb(\spkm); \mcalm)$ as the category of  cofunctors $F \colon \mcalbb(\spkm) \to \mcalm$ satisfying the following two conditions: 
	\begin{enumerate}
		\item[(a)] $F$ takes morphisms of  $\mcalw_{\Bb(\spkm)}$ to weak equivalences.
		\item[(b)] The image of every object under $F$ is fibrant.
	\end{enumerate}  
\end{defn}

\begin{rmk} \label{rmk:naturality}  For $F \in \mcalfb(\mcalbb(\spkm); \mcalm)$, and $V \in \mcalbb(\spkm)$, it is clear that the weak equivalence $F(\tau_V) \colon F(\theta \lambda (V)) \to F(V)$ is natural in $V$ since the morphisms of $\mcalbb(\spkm)$ are inclusions and since $F$ is a cofunctor. This is used implicitly in the proof of Lemma~\ref{lem:mcalbb} below. 
\end{rmk}

\begin{defn} \label{defn:mcalfba_spk}
	Let $\abb = \{A_0, \cdots, A_k\}$ be a sequence of objects of $\mcalm$. Define $\mcalfab (\mcalbb(\spkm); \mcalm)$ as the full subcategory of $\mcalfb(\mcalbb(\spkm); \mcalm)$ whose objects are cofunctors $F \colon \Bb(\spkm) \to \mcalm$ satisfying the following condition: for every $V \in \mcalbb(\spkm)$,  $F(V) \simeq A_i$ whenever $\deg(V) = i$. 
\end{defn}

We still need a couple of definitions before stating the first result of this section.

\begin{defn} \label{defn:isofun}
	Let $\mcals \subseteq \om$ be a subcategory of $\om$.  
	\begin{enumerate} 
		\item[(i)] A  cofunctor $F \colon \mcals \lra \mcalm$ is called \emph{isotopy cofunctor} if it satisfies the following two conditions:
		\begin{enumerate}
			\item[(a)] $F$ sends isotopy equivalences to weak equivalences. 
			\item[(b)]  For every $U \in \mcals$,  $F(U)$ is fibrant.  
		\end{enumerate}
		\item[(ii)] For $\mcals = \mcalb_k(M)$, the poset from Definition~\ref{defn:gb_mcalbk}, define $\mcalf(\mcalb_k(M); \mcalm)$ as the category of isotopy cofunctors from $\mcalb_k(M)$ to $\mcalm$. 
		\item[(iii)] For $A_0, \cdots, A_k \in \mcalm$, define $\mcalfa(\mcalb_k(M); \mcalm)$ as the category of isotopy cofunctors $F \colon \mcalb_k(M) \to \mcalm$ such that  $F(U) \simeq A_i$ whenever $U$ is the  union of $i$ disjoint elements of $\mcalb$. 
	\end{enumerate}  
\end{defn}

\begin{defn} \cite[Definition 6.3]{paul_don17-2} \label{defn:we}
	Let $\mcalc$ and $\mcald$ be categories both endowed with a class of maps called weak equivalences. 
	\begin{enumerate}
		\item[(i)] We say that two cofunctors  $F, G \colon \mcalc \lra \mcald$ are  \emph{weakly equivalent}, and we denote $F \simeq G$,  if they are connected by a zigzag of natural transformations which are objectwise weak equivalences. 
		\item[(ii)] A functor $F \colon \mcalc \lra \mcald$ is said to be a \emph{weak equivalence} if it satisfies the following two conditions.
		\begin{enumerate}
			\item[(a)] $F$ preserves weak equivalences.
			\item[(b)] There is another functor $G \colon \mcald \lra \mcalc$ such that $FG \simeq id$ and $GF \simeq id$.  The functor $G$ is also required to preserve weak equivalences. 
		\end{enumerate} 
		In that case we say that 	$\mcalc$ is \emph{weakly equivalent}  to $\mcald$.
	\end{enumerate}
\end{defn}

Now we want to show that the category $\mcalfb(\mcalbb(\spkm); \mcalm)$  is weakly equivalent to the category $\mcalf (\mcalb_k(M); \mcalm)$. To this end, define 
\[
\Psi \colon \mcalfb(\mcalbb(\spkm); \mcalm) \to \mcalf (\mcalb_k(M); \mcalm)
\]
 as $\Psi(F) = F \theta$. For every $A_0, \cdots, A_k \in \mcalm$, for every $F \in \mcalfab(\mcalbb(\spkm); \mcalm)$, one can easily check that $\Psi(F)$ lands in $\mcalfa (\mcalb_k(M); \mcalm)$.  This defines a functor
 \[
 \Psi_{\abb} \colon \mcalfab(\mcalbb(\spkm); \mcalm) \to \mcalfa (\mcalb_k(M); \mcalm)
 \]
 in the obvious way.

\begin{lem}  \label{lem:mcalbb}
	The functors $\Psi$ and $\Psi_{\abb}$ we just defined are both weak equivalences. 
\end{lem} 

\begin{proof}
	To see that $\Psi$ is a weak equivalence,  define  
	\[
	\Phi \colon \mcalf (\mcalb_k(M); \mcalm) \to \mcalfb(\mcalbb(\spkm); \mcalm)
	\]
	as $\Phi (G) = G \lambda$. 
	One has $\Psi \Phi = id$ since $\lambda \theta = id$ by Remark~\ref{rmk:lt=id}. Using Remark~\ref{rmk:naturality}, one can easily check that there is a map $id \to \Phi \Psi$ which is a weak equivalence. By the definitions, the functors $\Psi$ and $\Phi$ preserve weak equivalences. For the second part, the restriction functor $\Phi_{\abb}:= \Phi|\mcalfa(\mcalb_k(M); \mcalm)$ does the desired work. 
\end{proof}

We now want to relate the category $\mcalf(\mcalb_k(M); \mcalm)$ to the category of good polynomial cofunctors, which is defined as follows.

\begin{defn}   \label{defn:goodfun}
	A  cofunctor $F \colon \om \lra \mcalm$ is called \emph{good} if it satisfies the following two conditions:
	\begin{enumerate}
		\item[(a)] $F$ is an isotopy cofunctor (see Definition~\ref{defn:isofun});
		\item[(b)] For any string $U_0 \ra U_1 \ra \cdots$ of inclusions of $\om$, the natural map 
		\[
		F\left(\bigcup_{i=0}^{\infty} U_i\right) \lra \underset{i}{\text{holim}} \; F(U_i)
		\]
		is a weak equivalence. 
	\end{enumerate} 
\end{defn}

\begin{defn} \label{defn:PF}
	A  cofunctor $F \colon \om \lra \mcalm$ is called \emph{polynomial of degree $\leq k$} if for every $U \in \om$ and pairwise disjoint closed subsets $C_0, \cdots, C_k$ of $U$, the canonical map 
	\[
	F(U) \lra \underset{S \neq \emptyset}{\text{holim}} \; F(U \backslash \cup_{i \in S} C_i)
	\] 
	is a weak equivalence. Here $S \neq \emptyset$ runs over the power set of $\{0, \cdots, k\}$. 
\end{defn}

\begin{defn} \label{defn:mcalpa}
	\begin{enumerate} 
		\item[$\bullet$] Define $\mcalp_k(\om; \mcalm)$ as the category of  good cofunctors $F \colon \om \to \mcalm$ that are polynomial of degree $\leq k$ (see Definitions~\ref{defn:goodfun} and \ref{defn:PF}). 
		\item[$\bullet$] For $A_0, \cdots, A_k \in \mcalm$, define $\mcalpa(\om; \mcalm)$ as the subcategory of $\mcalp_k(\om; \mcalm)$ whose objects are cofunctors $F \colon \om \to \mcalm$ such that $F(U) \simeq A_i$ for every $U$ diffeomorphic to the disjoint union of exactly $i$ open balls (not necessarily in $\mcalb$). 
	\end{enumerate} 
\end{defn}

\begin{lem}  \label{lem:PF} 
	\begin{enumerate}
		\item[(i)] The restriction functor 
		\[
		res \colon \mcalp_k(\om; \mcalm) \to \mcalf(\mcalb_k(M); \mcalm), \quad F \mapsto F|\mcalb_k(M),
		\]
	   is a weak equivalence.  
		\item[(ii)] For every $A_0, \cdots, A_k \in \mcalm$, the same statement holds for $\mcalpa(\om; \mcalm)$ and $\mcalfa(\mcalb_k(M); \mcalm)$. 
	\end{enumerate} 
\end{lem}

\begin{proof}
	For the first part, 
	define $(-)^{!} \colon \mcalf(\mcalb_k(M); \mcalm) \to \mcalp_k(\om; \mcalm)$ as the homotopy right Kan extension, that is, $G^!(U):= \underset{V \in \mcalb_k(U)}{\text{holim}} \, F(V)$. For any  $G \in \mcalf(\mcalb_k(M); \mcalm)$, the cofunctor $G^!$ is good (by \cite[Theorem 4.2]{paul_don17-2}) and polynomial of degree $\leq k$ (by \cite[Lemma 5.3]{paul_don17-2}). So $(-)^!$ lands in the right category.  By definition, the canonical map $id \to res (-)^!$ is a weak equivalence. By \cite[Theorem 1.1]{paul_don17-2}, the canonical map $id \to (-)^!  res$ is also a weak equivalence. Clearly, the functors $res$ and $(-)^!$ preserve weak equivalences.  This proves the first part. Likewise, one can prove the second part by considering the restriction functors $res |\mcalpa(\om; \mcalm)$ and $(-)^!|\mcalfa(\mcalb_k(M); \mcalm$. 
\end{proof} 

Combining  Lemmas~\ref{lem:mcalbb} and \ref{lem:PF}, we have the following.  

\begin{prop}  \label{prop:mcalpk_mcalfb}
	Consider the functors $\Phi, \Phi_{\abb}$, and $res$ from Lemmas~\ref{lem:mcalbb} and \ref{lem:PF}.  Then the composites 
	\[
	\Phi res \colon \mcalp_k(\om; \mcalm) \to \mcalfb(\Bb(\spkm); \mcalm)    \text{ and } \Phi_{\abb} res \colon \mcalp_{k\abb}(\om; \mcalm) \to \mcalfb_{\abb}(\Bb(\spkm); \mcalm)
	\]
	 are both weak equivalences. 
\end{prop}


\section{The category of linear functors from $\mcalo(\spkm)$ to $\mcalm$}  \label{sec:linear_func}

In this section  we introduce a version of linear cofunctors out of   $\mcalo(\spkm)$.   Such cofunctors are simply defined to be the homotopy right Kan extension of cofunctors  $\Bb(\spkm) \to \mcalm$, where $\Bb(\spkm) \subseteq \mcalo(\spkm)$ is the full subposet from Definition~\ref{defn:spkm} which is constructed out of $\mcalb$. The main result of this section, Theorem~\ref{thm:lb_lbp} below, says that this definition is independent of the choice of  $\mcalb$. Essentially we use the same approach as  Pryor-Weiss \cite{pryor15, wei99} to prove this.

\subsection{Definition of  $\mcall^{\mcalb}(\mcalo(\spkm); \mcalm)$}

We begin with a few definitions.

\begin{defn}
	\begin{enumerate}
		\item[(i)] 	For an open subset $U$ of $\spkm$, define $\mcalbb(U) \subseteq \Bbspkm$ as the full subposet $\mcalbb(U) = \{V \in \Bbspkm| \ V \subseteq U\}.$  
		\item[(ii)]  Given a cofunctor $F \colon \mcalbb(\spkm) \to \mcalm$, we let $F^{!} \colon \mcalo (\spkm) \to \mcalm$ denote its homotopy right Kan extension  along the inclusion $\mcalbb(\spkm) \hra \mcalo(\spkm)$. Explicitly, 
		\begin{equation} \label{eqn:fsrik}
		F^{!}(U):= \underset{V \in \mcalbb(U)}{\text{holim}} \, F(V). 
		\end{equation}
	\end{enumerate}
\end{defn}

\begin{defn} \label{defn:blinear}
	A cofunctor $F \colon \mcalo (\spkm) \to \mcalm$ is called \emph{$\mcalb$-linear} if the canonical map $F \to (F|\mcalbb(\spkm))^{!}$ is a weak equivalence. 
\end{defn}

Recall the categories $\mcalfb (\mcalbb(\spkm); \mcalm)$ and $\mcalfab (\mcalbb(\spkm); \mcalm)$ from Definitions~\ref{defn:mcalfb_spk} and \ref{defn:mcalfba_spk}. 

\begin{defn} \label{defn:bgood} Let $F \colon \mcalo(\spkm) \to \mcalm$ be a  cofunctor.
	\begin{enumerate}
		\item[(i)] The cofunctor $F$ is called \emph{$\mcalb$-good} if it satisfies the following two conditions:
		\begin{enumerate}
			\item[(a)] For every $U \in \mcalo(\spkm)$, $F(U)$ is fibrant.
			\item[(b)] The restriction $F|\mcalbb(\spkm)$ belongs to $\mcalfb (\mcalbb(\spkm); \mcalm)$.
		\end{enumerate} 
		\item[(ii)] For $\abb= \{A_0, \cdots, A_k\} \in \mcalm^{k+1}$, $F$ is said to be \emph{$\mcalb_{\abb}$-good} if it satisfies (a) and $F|\mcalbb(\spkm) \in \mcalfab (\mcalbb(\spkm); \mcalm$. 
	\end{enumerate}	 
\end{defn} 

Now we can define the category $\mcall^{\mcalb}(\mcalo(\spkm); \mcalm)$. 

\begin{defn} \label{defn:mcallb}  Let $\mcalb$ be a good basis for the topology of $M$ (see Definition~\ref{defn:gb_mcalbk}), and let  $\mcalbb(\spkm)$ be the poset from Definition~\ref{defn:spkm}. 
	\begin{enumerate} 
		\item[(i)] Define $\mcall^{\mcalb} (\mcalo(\spkm); \mcalm)$ as the category of  cofunctors $F \colon \mcalo(\spkm) \to \mcalm$ that are $\mcalb$-good and $\mcalb$-linear (see Definitions~\ref{defn:bgood} and \ref{defn:blinear}).
		\item[(ii)] For $\abb= \{A_0, \cdots, A_k\} \in \mcalm^{k+1}$, define $\mcall^{\mcalb}_{\abb} (\mcalo(\spkm); \mcalm)$ as the category of  cofunctors $F \colon \mcalo(\spkm) \to \mcalm$ that are $\mcalb_{\abb}$-good and $\mcalb$-linear.
	\end{enumerate} 
\end{defn}

The next section addresses the question to know whether the category $\mcall^{\mcalb}(\mcalo(\spkm); \mcalm)$ depends on $\mcalb$.

\subsection{Dependence of $\mcall^{\mcalb} (\mcalo(\spkm); \mcalm)$ on the choice of $\mcalb$}  \label{subsec:dependence}

The goal of this section is to prove Theorem~\ref{thm:lb_lbp}, which says that the category $\mcall^{\mcalb} (\mcalo(\spkm); \mcalm)$ does not depend on $\mcalb$. We need a number of intermediate results. Most of them involve the geometric realization functor which we denote $|-| \colon \text{Cat} \to \text{Top}$ as usual. 

\begin{lem}  \label{lem:coverX}
	Let $X$ be a topological space, and let $\mcalu = \{U_i\}_{i \in I}$ be a cover of $X$. ($\mcalu$ need not be an open cover.) Assume that $\mcalu$ has the following properties.
	\begin{enumerate}
		\item[(a)]  Each $U_i$ is contractible.
		\item[(b)]  For every $i, j \in I$ such that $U_i \cap U_j \neq \emptyset$, for every $x \in U_i \cap U_j$, there exists $k \in I$ such that $x \in U_k \subseteq U_i \cap U_j$.
	\end{enumerate} 
	Then the geometric realization of $\mcalu$, viewed as a poset ordered by inclusion, is homotopy equivalent to $X$. That is, $|\mcalu| \simeq X$. 
\end{lem} 

\begin{proof}
	This works in the same way as the proof of Lemma 3.5 from \cite{wei99}. 
\end{proof} 

To state the next results, we need to make the following definition.

\begin{defn} \label{defn:Bbp}
	Let $p \geq 0$ be an integer, and let $U$ be an open subset of $\spkm$. 
	\begin{enumerate}
		\item[(i)] Define $\Bb_p(U)$ as the poset whose objects are strings of $p$ composable morphisms in $\Bb(U)$, $V_0 \to \cdots \to V_p$, and whose morphisms are natural transformations of such diagrams whose component maps lie in $\mcalw_{\Bb(\spkm)}$ (see Definition~\ref{defn:mcalw_Bb}).   
		\item[(ii)] Define $\Bt_p(U) \subseteq  \Bb_p(U)$ as the full subposet whose objects are strings  $V_0 \to \cdots \to V_p$ such that every $V_i$ is of the form (\ref{eqn:type1}). 
	\end{enumerate} 
\end{defn} 

\begin{prop} \label{prop:Bbzu}
	For every open subset $U \subseteq \spkm$, one has $|\Bt_0(U)| \simeq U$. 
\end{prop}

\begin{proof}
	Let $U$ be an open subset of  $\spkm$. Since $\Bt_0(U) = \coprod_{r=1}^{k} \Bt_0^{(r)}(U)$, where $\Bt_0^{(r)}(U) \subseteq \Bt_0(U)$ is the full subposet of objects of degree $r$ (see Definition~\ref{defn:degree}), it suffices to show that 
	\[
	|\Bt_0^{(r)}(U)| \simeq U \cap \left(\Sp^r(M) \backslash \Sp^{r-1}(M)\right) =: X.
	\]
	
	Let
	\[
	\widetilde{\mcalb}_0^{(r)}(U) \cap X := \left\{ V \cap X| \; V \in \widetilde{\mcalb}_0^{(r)}(U)\right\}.
	\]
	One can see that this is a cover of $X$, and satisfies conditions (a) and (b) of Lemma~\ref{lem:coverX}. Therefore, one has $|\widetilde{\mcalb}_0^{(r)}(U) \cap X| \simeq X$. Furthermore, it is clear that the poset $\widetilde{\mcalb}_0^{(r)}(U) \cap X$ is isomorphic to $\widetilde{\mcalb}_0^{(r)}(U)$. Hence,
	\[
	|\widetilde{\mcalb}_0^{(r)}(U)| \cong  |\widetilde{\mcalb}_0^{(r)}(U) \cap X| \simeq X, 
	\]
	as we needed to show. 
\end{proof}  

\begin{lem} \label{lem:BtBb}
	For every open subset $U \subseteq \spkm$, the map $|\Bt_p(U)| \to |\Bb_p(U)|$ induced by the inclusion functor $i \colon \Bt_p(U) \hra \Bb_p(U)$ is a homotopy equivalence.  
\end{lem} 

\begin{proof}
	Let $(V_0 \to \cdots \to V_p) \in \Bb_p(U)$. By Quillen's Theorem A, it suffices to show that the over category  $i \downarrow (V_0 \to \cdots \to V_p)$ is contractible, and it is since it has a terminal object, namely $\widehat{V_0} \to \cdots \to \widehat{V_p}$. (See Remark~\ref{rmk:objects_Bb} for the definition of $\widehat{(-)}$.) 
\end{proof}   

The following lemma is the analog of \cite[3.6 and 3.7]{wei99} and \cite[Lemma 6.8]{pryor15}. To prove it we  use the same approach as Weiss, but with different categories. For our purposes, we need to recall the Grothendieck construction. Let $\varphi \colon \mcalc \to \textup{Cat}$ be a functor from a small category to the category of small categories. Associated with $\varphi$ is a category $\int_{\mcalc} \varphi$, called the \textit{Grothendieck construction}, and defined as follows. Objects are pairs $(c, x)$ where $c \in \mcalc$ and $x \in \varphi(c)$. A morphism $(c, x) \to (c', x')$ consists of a pair $(f, g)$ where $f \colon c \to c'$ is a morphism of $\mcalc$ and $g \colon \varphi(f)(x) \to x'$ is a morphism of $\varphi(c')$.

\begin{lem}  \label{lem:hfiber}
	Let $U$ be an open subset of $\spkm$. Let $W \in \Bt_0(U)$ and let $x_W \in |\Bt_0(U)|$ be the $0$-simplex corresponding to $W$. Let $\mcalb'$ be another good basis  for the topology of $M$.
	\begin{enumerate}
		\item[(i)]  Then the homotopy fiber over $x_W$ of the map $|\Bt_p(U)| \to |\Bt_0(U)|$, induced by $(V_0 \to \cdots \to V_p) \mapsto V_p$, is homotopy equivalent to $|\Bt_{p-1}(W)|$. 
		\item[(ii)] Let $f \colon V \to V'$ be a morphism of $\Bpb(\spkm)$ such that $\deg(V) = \deg(V')$ and $\pi_0(f)$ is a bijection. Then the map $|\Bt_p(f)| \colon |\Bt_p(V)| \to |\Bt_p(V')|$ is a homotopy equivalence.  
		\item[(iii)] If $\mcalb \subseteq \mcalb'$, the canonical  map $|\Bt_p(U)| \to |\widetilde{\mcalb'}_p(U)|$  is a homotopy equivalence.  
	\end{enumerate}  
\end{lem}

\begin{proof}
	We begin with the first part. We will  prove by induction on $p$ that the homotopy fiber of the map $|\Bt_p(U)| \to |\Bt_0(U)|$ over $x_W$ is $|\Bt_{p-1}(W)|$, and the functor $|\Bt_p(-)| \colon \Bt_0 (U) \to \textup{Top}$ takes all morphisms to homotopy equivalences. For $p = 0$, the functor $|\Bt_0(-)|$ has the required property by Proposition~\ref{prop:Bbzu}. 
	Now for $p > 0$, consider the functor $\varphi \colon \Bt_0(U) \to \textup{Cat}$ defined as $\varphi(V) = \Bt_{p-1}(V)$,  and the
	Grothendieck construction $\int_{\Bt_0(U)} \varphi$. The functor
	\[
	\int_{\Bt_0(U)} \varphi \to \Bt_p(U),  \quad (V, V_0 \to \cdots \to V_{p-1}) \mapsto (V_0 \to \cdots \to V_{p-1} \to V)
	\]
	 is an isomorphism of categories with inverse $(V_0 \to \cdots \to V_p) \mapsto (V_p, V_0 \to \cdots \to V_{p-1})$. Moreover, by the Thomason homotopy colimit theorem \cite{thomason79}, one has a natural homotopy equivalence $\underset{V \in \Bt_0(U)}{\textup{hocolim}} \, |\Bt_{p-1}(V)| \simeq |\int_{\Bt_0(U)} \varphi|$. So 
	\[
	|\Bt_p(U)| \simeq \underset{V \in \Bt_0(U)}{\textup{hocolim}} \, |\Bt_{p-1}(V)|.
	\]
	Since $|\Bt_0(U)| \simeq \underset{V \in \Bt_0(U)}{\textup{hocolim}} \, *$, one can identified the map $|\Bt_p(U)| \to |\Bt_0(U)|$ we are interested in as the canonical map 
	\[
	\underset{V \in \Bt_0(U)}{\textup{hocolim}} \, |\Bt_{p-1}(V)| \to \underset{V \in \Bt_0(U)}{\textup{hocolim}} \, *.
	\]
	Since the functor $|\Bt_{p-1}(-)| \colon \Bt_0(U) \to \textup{Top}$ takes all morphisms to homotopy equivalences by the induction hypothesis, it follows that this latter map is a quasifibration by a result of Quillen  \cite{quillen75}. Therefore its homotopy fiber is homotopy equivalent to its actual fiber over $x_{W}$, which can be taken to be $|\Bt_{p-1}(W)|$. 
	
	Now we need to prove that the functor $|\Bt_p(-)| \colon \Bt_0(U) \to \textup{Top}$ sends every morphism to a homotopy equivalence. Let $f \colon V \to V'$ be a morphism of  $\Bt_0(U)$. Consider the following map of quasifibration sequences.
	\[
	\xymatrix{|\Bt_{p-1}(W)| \ar[r] \ar[d] & |\Bt_p(V)| \ar[r]  \ar[d]^-{\Bt_p(f)} & |\Bt_0(V)| \ar[d]^-{\Bt_0(f)} \\
		|\Bt_{p-1}(W)|  \ar[r]  &  |\Bt_p(V')|  \ar[r]   &   |\Bt_0(V')|}
	\]
	Since the righthand vertical map is a homotopy equivalence by the base case, and since the lefthand vertical map can be taken to be the identity, it follows that the middle vertical map is homotopy equivalence as well. This proves the first part.
	
	Using (i) together with Proposition~\ref{prop:Bbzu}, one can easily prove (ii) and (iii) by induction on $p$.   
\end{proof}	

\begin{lem}  \label{lem:hfiber2}
	Let $U$ be an open subset of $\spkm$. 
	\begin{enumerate}
		\item[(i)] Let $f \colon V \to V'$ be as in Lemma~\ref{lem:hfiber}. Then the map $\Bb_p(f) \colon |\Bb_p(V)| \to |\Bb_p(V')|$ is a homotopy equivalence. 
		\item[(ii)] If $\mcalb'$ is another good basis containing $\mcalb$,  the  map $|\Bb_p(U)| \to |\Bpb_p(U)|$  is a homotopy equivalence. 
	\end{enumerate}
\end{lem}

\begin{proof}
	This follows immediately from Lemmas~\ref{lem:BtBb} and \ref{lem:hfiber}. 
\end{proof}

The proof of the main result of this section is based on the fact that $F^!$, the homotopy right Kan extension of  $F \colon \Bb(\spkm) \to \mcalm$, can be written as the homotopy limit of a certain  diagram of cofunctors, $F^{!p}$, defined as follows.

\begin{defn}  \label{defn:fsrikp}
	Let $F \colon \Bb(\spkm) \to \mcalm$ be a  cofunctor. Define a new cofunctor $F^{!p} \colon \mcalo(\spkm) \to \mcalm$ as 
	\begin{equation} \label{eqn:fsrikp}
	F^{!p}(U) = \underset{(V_0 \to \cdots \to V_p) \in \Bb_p(U)}{\textup{holim}} \, F(V_0).
	\end{equation}
\end{defn} 

\begin{lem}  \label{lem:tot}
	For every cofunctor $F \colon \Bb(\spkm) \to \mcalm$, $F^{!} \colon \mcalo(\spkm) \to \mcalm$ is weakly equivalent to the homotopy totalization of the cosimplicial object $[p] \mapsto F^{!p}$. That is, 
	\[
	F^! \simeq \textup{Tot} \left([p] \mapsto F^{!p}\right). 
	\] 
\end{lem}

\begin{proof}
	This goes in the same way as the proof of \cite[Lemma 6.10]{pryor15}
\end{proof}

Another lemma we need is the following. 

\begin{lem} \cite[Lemma 4.16]{paul_don17-2}  \label{lem:holim_we}
	Let $\iota \colon \mcalc \to \mcald$ be a functor between small categories. Let $F \colon \mcald \to \mcalm$ be a cofunctor that sends every morphism to a weak equivalence. Assume that that nerve of $\iota$ is a weak equivalence. Then the canonical map $\underset{\mcald}{\textup{holim}}\, F \to \underset{\mcalc}{\textup{holim}}\, F\iota$ is a weak equivalence. 
\end{lem}

Now we can state and prove the main result of the section. 

\begin{thm} \label{thm:lb_lbp}
	Let $\mcalb'$ be another good basis for the topology of $M$. 
	\begin{enumerate} 
		\item[(i)] Then the category  $\mcall^{\mcalb} (\mcalo(\spkm); \mcalm)$ of $\mcalb$-good and $\mcalb$-linear cofunctors (see Definition~\ref{defn:mcallb}) is the same as the category $\mcall^{\mcalb'} (\mcalo(\spkm); \mcalm)$. That is,
		\[
		\mcall^{\mcalb} (\mcalo(\spkm); \mcalm)   = \mcall^{\mcalb'} (\mcalo(\spkm); \mcalm). 
		\]
		\item[(ii)] For every $A_0, \cdots, A_k \in \mcalm$, one has  
		\[
		\mcall_{\abb}^{\mcalb} (\mcalo(\spkm); \mcalm)   = \mcall_{\abb}^{\mcalb'} (\mcalo(\spkm); \mcalm). 
		\]
	\end{enumerate}	 
\end{thm}

\begin{proof}
	Since $\mcalb \cup \mcalb'$ is again a good basis for the topology of $M$, we can assume without loss of generality that $\mcalb \subseteq \mcalb'$. We want to show that $\mcall^{\mcalb} (\mcalo(\spkm); \mcalm)   \subseteq  \mcall^{\mcalb'} (\mcalo(\spkm); \mcalm)$. So  let $F \in \mcall^{\mcalb} (\mcalo(\spkm); \mcalm)$. We need to show that $F$ is $\mcalb'$-good and $\mcalb'$-linear. 
	
	$\bullet$ Proving that $F$ is $\mcalb'$-good. By definition, we have to show that $F$ satisfies the following two conditions: 
	\begin{enumerate}
		\item[-] $F(V)$ is fibrant for every $V \in \mcalo(\spkm)$.
		\item[-]  $F(f) \colon F(V') \to F(V)$ is a weak equivalence whenever $f \colon V \to V'$ is a morphism of $\mcalw_{\Bpb (\spkm)}$. 
	\end{enumerate}
	The first condition holds since  $F$ is $\mcalb$-good. For the second, let $f \colon V \to V'$ be a morphism of $\mcalw_{\Bpb (\spkm)}$. We need to deal with three cases. 
	\begin{enumerate}
		\item[-] Suppose $\deg(V) = \deg(V')$ and $\pi_0(f)$ is a bijection.  
		Since $F$ is $\mcalb$-linear, the canonical map $F \to (F|\Bb(\spkm))^{!}$ is a weak equivalence. So it suffices to show that the map $(F|\Bb(\spkm))^{!}(f)$ is a weak equivalence. And by Lemma~\ref{lem:tot}, it is enough to  show that the canonical map $F^{!p}(f) \colon  F^{!p}(V') \to F^{!p}(V)$ is a weak equivalence for all $p$. By (\ref{eqn:fsrikp}), this latter map can be rewritten as 
		\begin{equation} \label{eqn:fuup}
		\underset{(V_0 \to \cdots \to V_p) \in \Bb_p(V')}{\textup{holim}} \, F(V_0) \to \underset{(V_0 \to \cdots \to V_p) \in \Bb_p(V)}{\textup{holim}} \, F(V_0).
		\end{equation}
		Since the cofunctor $\Bb_p(V') \to \mcalm, (V_0 \to \cdots \to V_p) \mapsto F(V_0)$ takes all morphisms to weak equivalences (because $F$ is $\mcalb$-good), and since the geometric realization of the inclusion functor  $\Bb_p(V) \hra \Bb_p(V')$ is a weak equivalence (by Lemma~\ref{lem:hfiber2}), it follows that (\ref{eqn:fuup}) is a weak equivalence by Lemma~\ref{lem:holim_we}, as we needed to show. 
		\item[-] Suppose $V' = \theta \lambda (V)$ with $V$ of the form (\ref{eqn:type1}).  Since $\mcalb$ is a basis for the topology of $M$, there exists $W \subseteq V$, $W \in \Bb(\spkm)$, such that the inclusion $W \hra V$ has the following properties: $\deg(V) = \deg(W)$ and $\pi_0(W \hra V)$ is a bijection. The objects $V, \theta \lambda (V), W,$ and $\theta \lambda (W)$ fit into the following commutative square. 
		\[
		\xymatrix{F (V) \ar[d] &  F(\theta \lambda (V)) \ar[l] \ar[d] \\
			F(W)     &    F(\theta \lambda (W)) \ar[l]}
		\]
		The vertical maps are both weak equivalences by the first case, and the bottom map is a weak equivalence since $F$ is $\mcalb$-good. Therefore, so is the top map.  
		\item[-] Now suppose $f$ factors as $V \stackrel{g}{\to} W \stackrel{h}{\to} V'$ where $g$ and $h$ are as in Definition~\ref{defn:mcalw_Bb}. Then $F(g)$ is a weak equivalence by the first case, and $F(h)$ is a weak equivalence as well by the second case. Therefore $F(f) = F(g)F(h)$ is a weak equivalence. 
	\end{enumerate}
	
	$\bullet$ Proving that $F$ is $\mcalb'$-linear. By the commutative triangle
	\[
	\xymatrix{F \ar[r]^-{\sim} \ar[rd]  &  (F|\Bb(\spkm))^{!}  \\
		&   (F|\Bpb(\spkm))^{!} \ar[u],}
	\]
	where each map is the canonical one, it is suffices to show that the vertical map is a weak equivalence, and it is by using Lemmas~\ref{lem:tot}, \ref{lem:hfiber2} and \ref{lem:holim_we}, and  arguing in a similar way as before. 
	
	Similarly, one has $\mcall^{\mcalb'} (\mcalo(\spkm); \mcalm)   \subseteq  \mcall^{\mcalb} (\mcalo(\spkm); \mcalm)$. The second part part of the theorem can be proved in  similar fashion. 
	
\end{proof} 

Thanks to Theorem~\ref{thm:lb_lbp}, we can drop the superscript $\mcalb$ and make the following definition.

\begin{defn}  \label{defn:mcalla}
	Define $\mcall(\mcalo(\spkm); \mcalm)$ (respectively $\mcall_{\abb}(\mcalo(\spkm); \mcalm)$) as the category $\mcall^{\mcalb}(\mcalo(\spkm); \mcalm)$ from Definition~\ref{defn:mcallb} (respectively  $\mcall_{\abb}^{\mcalb}(\mcalo(\spkm); \mcalm)$). 
\end{defn}

\section{Proof of the main result}  \label{sec:main_thm}

We still need one intermediate result before proving Theorem~\ref{thm:main}.

\begin{lem}  \label{lem:mcall_mcalfab}
	\begin{enumerate}
		\item[(i)] The restriction functor 
		\[
		res \colon \mcall (\mcalo(\spkm); \mcalm) \to \mcalfb (\mcalbb(\spkm); \mcalm)
		\]
		is a weak equivalence. 
		\item[(ii)] As usual, for every $A_0, \cdots, A_k  \in \mcalm$, the same statement holds for $\mcall_{\abb} (\mcalo(\spkm); \mcalm)$ and $\mcalfab (\mcalbb(\spkm); \mcalm)$. 
	\end{enumerate}	
\end{lem}

\begin{proof}
	To prove the first part, 
	let 
	\[
	(-)^! \colon \mcalfb (\mcalbb(\spkm); \mcalm) \to \mcall (\mcalo(\spkm); \mcalm)
	\]
	 be the functor defined by (\ref{eqn:fsrik}). It is straightforward to check that for every $G$, the cofunctor $G^!$ lands in $\mcall (\mcalo(\spkm); \mcalm)$.	
	It is also straightforward to check that the canonical maps $id \to (-)^! res$ and $id \to res (-)^!$ are both weak equivalences. Clearly, the functors $res$ and $(-)^!$ preserve weak equivalences.  For the second part, the restriction  $(-)^!|\mcalfab (\mcalbb(\spkm); \mcalm)$ does the desired work. 
\end{proof}

Now we are ready to prove the main result of this paper. 
%

\begin{proof}[Proof of Theorem~\ref{thm:main}]
	Consider the  diagram
	\begin{equation} \label{eqn:hl0}
	\xymatrix{\mcalf(\mcalb_k(M); \mcalm) \ar@<1ex>[rrr]^-{\Phi} \ar@<1ex>[d]^{(-)^!} & & & \mcalfb(\mcalbb(\spkm); \mcalm) \ar@<1ex>[lll]^-{\Psi} \ar@<1ex>[d]^{(-)^!} \\
		\mcalp_k(\om; \mcalm) \ar@<1ex>[u]^{res}  & &  &   \mcall(\mcalo(\spkm); \mcalm), \ar@<1ex>[u]^{res}  }
	\end{equation}
	where  
	\begin{enumerate}
		\item[$\bullet$] $\mcalfb(\mcalbb(\spkm); \mcalm)$, $\mcalf(\mcalb_k(M); \mcalm)$, $\mcalp_k(\om; \mcalm)$, and $\mcall(\mcalo(\spkm); \mcalm)$ are the categories from Definitions~\ref{defn:mcalfb_spk}, \ref{defn:isofun}, \ref{defn:mcalpa}, and \ref{defn:mcalla}, 
		\item[$\bullet$] $\Phi$ and $\Psi$ are the functors of Lemma~\ref{lem:mcalbb}. To recall,  $\Phi (G)= G \lambda$ and $\Psi(F)= F \theta$, where $\lambda \colon \Bb(\spkm) \to \mcalb_k(M)$ and $\theta \colon \mcalb_k(M) \to \Bb(\spkm)$ are the functors from Definition~\ref{defn:lamb_theta}, and 
		\item[$\bullet$]  $res$ and $(-)^!$ are the usual restriction and homotopy right Kan extension functors (see the proofs of Lemmas~\ref{lem:PF} and \ref{lem:mcall_mcalfab}). 
	\end{enumerate}
	The functors
	\begin{equation} \label{eqn:theta_lambda}
	    \xymatrix{\mcalp_k(\om; \mcalm) \ar@<1ex>[r]^-{\Theta} & \mcall(\mcalo(\spkm); \mcalm) \ar@<1ex>[l]^-{\Lambda}}
	\end{equation}
	that appear in the introduction are given by $\Theta = (-)^! \Phi res$  and $\Lambda = (-)^!\Psi res$. Combining Lemmas~\ref{lem:mcalbb}, \ref{lem:PF} and \ref{lem:mcall_mcalfab}, we have that the canonical maps $id \to \Lambda \theta$ and $id \to \Theta \Lambda$ are both weak equivalences. The same lemmas imply that $\Theta$ and $\Lambda$ preserve weak equivalences. This proves the theorem. 
\end{proof}


%
%
%

\section{Homogeneous layer of functors out of $\mcalo(\spkm)$}  \label{section:hl}

Suppose $\mcalm$ has a zero object, $*$, and let $G \colon \om \to \mcalm$ be a cofunctor. 
For $r \geq 0$, the $r$th \textit{polynomial approximation} to $G$, denoted $T_r(G) \colon \om \to \mcalm$, can be defined as $T_r (G) (U):= \underset{W \in \mcalb_r(U)}{\text{holim}} G(W)$. The cofunctor $G$ is called \textit{homogeneous of degree $r$} if $G$ is polynomial of degree $\leq r$ (see Definition~\ref{defn:PF}) and  $T_{r-1} (G) \simeq *$. It is well known that the cofunctor $L_r(G) \colon \om \to \mcalm$ defined as 
\[
L_r(G)(U):= \hfib\left( T_r(G)(U) \to  T_{r-1}(G)(U)\right)
\] 
is homogeneous of degree $r$. The cofunctor $L_r(G)$ is commonly called the \textit{$r$th homogeneous layer} of $G$. The goal of this section is to define the concept of homogeneous layer for cofunctors from $\mcalo(\spkm)$ to $\mcalm$. Specifically, let $\Theta$ and $\Lambda$ be the functors of (\ref{eqn:theta_lambda}). 
Given $F \in \mcall(\mcalo(\spkm); \mcalm)$ and $r \leq k$, we will define a new cofunctor $\Lb_r(F) \colon \mcalo(\spkm) \to  \mcalm$ and prove Theorem~\ref{thm:hl} below, which roughly says  that $\Lambda(\Lb_r(F)) \simeq L_r(\Lambda (F))$. (Our definition of $\Lb_r(F)$ (see Definition~\ref{defn:Lb} below) does not involve $\Theta$ or $\Lambda$.) 



From now on,  $F$ is an object of the category  $\mcall(\mcalo(\spkm); \mcalm)$ introduced in Definition~\ref{defn:mcalla}, and $r \leq k$. Consider the homogeneous cofunctor of degree $r$, $L_r(\Lambda (F)) \colon \om \to \mcalm$. Explicitly, 
\begin{equation} \label{eqn:hl01}
L_r(\Lambda(F)) = L_r(\Psi (res(F))^!) = \hfib\left( T_r(\Psi(F)^!) \to T_{r-1}(\Psi(F)^!)\right).
\end{equation}

On the other hand, we need to define $\Lb_r(F)$. First of all, consider the poset $\Bb_0(\spkm)$ from Definition~\ref{defn:Bbp}. Let $\Bb_0^{(r)}(\spkm) \subseteq \Bb_0(\spkm)$ be the full subposet whose objects are those of degree $r$ (see Definition~\ref{defn:degree}). For simplicity, we will write $\Bb^{(r)}(\spkm)$ for $\Bb_0^{(r)}(\spkm)$. 

\begin{defn} \label{defn:Lb}  
	\begin{enumerate}
		\item[(i)] Define a cofunctor $L'_r(F) \colon \Bb^{(r)} (\spkm) \to \mcalm$ as
		\[
		L'_r(F)(V):= \hfib\left( F(\theta \lambda (V) \to F(\theta \lambda (V) \backslash \widehat{V})\right),
		\]
		where $\widehat{V}$ is defined as in Remark~\ref{rmk:objects_Bb}. 
		\item[(ii)] Extend this to a cofunctor $L''_r (F) \colon \Bb(\spkm) \to \mcalm$ defined as $L''_r(F)(V):= L'(F)(V)$ if $V \in \Bb^{(r)} (\spkm)$. Otherwise, $L''_r(F)(V):= *$. For a morphism $f \colon V \to W$ of $\Bb(\spkm)$, define $L''_r(f):= L'_r(f)$ if $f$ is a morphism of $\Bb^{(r)}(\spkm)$. Otherwise, define $L''(F)(f)$ to be the zero morphism. 
		\item[(iii)] Now define $\Lb_r(F)$ as the homotopy right Kan extension of $L''_r(F)$ along the inclusion $\Bb(\spkm) \hra \mcalo (\spkm)$. That is,
		\[
		\Lb_r(F)(V):= \underset{W \in \Bb(V)}{\text{holim}} L''_r(F)(W).
		\]
	\end{enumerate}
\end{defn}

The following is straightforward.

\begin{lem} \label{lem:hl-Lbr}
	The cofunctor $\Lb_r(F)$ belongs to $\mcall(\mcalo(\spkm); \mcalm)$. Moreover, the canonical map $L'_r(F) \to \Lb_r(F)|\Bb^{(r)}(\spkm)$ is a weak equivalence.
\end{lem}


To define the second homogeneous cofunctor  $\om \to \mcalm$ of degree $r$ we are interested in,  let $\mcalb^{(r)}(M) \subseteq \mcalb_k(M)$ be the subposet whose objects are unions $U = B_1 \cup \cdots \cup B_r$ of exactly  $r$ pairwise disjoint elements from  $\mcalb$, and whose morphisms are isotopy equivalences. In \cite[Lemma 6.5]{paul_don17-2} we prove that homogeneous cofunctors $\om \to \mcalm$ of degree $r$ are determined by their values on   $B^{(r)}(M)$. Specifically, we construct two functors
\[
\xymatrix{\mcalf(\mcalb^{(r)}(M); \mcalm) \ar@<1ex>[r]^-{\psi}   &  \mcalh_r(\om; \mcalm) \subseteq \mcalp_k(\om; \mcalm) \ar@<1ex>[l]^-{\phi}}
\]
between the category of isotopy cofunctors $\mcalb^{(r)}(M) \to \mcalm$ and that of good homogeneous cofunctors of degree $r$, and we show the following result.

\begin{lem} \cite[Lemma 6.5]{paul_don17-2} \label{lem:hl-phipsi}
	The functors $\phi$ and $\psi$ preserve weak equivalences. Furthermore, the canonical maps $id \to \phi \psi$ and $id \to \psi\phi$ are both weak equivalences.  
\end{lem}

In fact, $\phi$ is just the restriction functor, and $\psi$ is a bit more subtle. For $G \in \mcalf(\mcalb^{(r)}(M); \mcalm)$,  $\psi(G)(U) = \underset{V \in \mcalb_r(U)}{\text{holim}} G'(V)$, where $G' \colon \mcalb_r(M) \to \mcalm$ is the cofunctor given by $G'(V) = G(V)$ if $V \in \mcalb^{(r)}(M)$, and $*$ otherwise. For a morphism $f$ of  $\mcalb_r(M)$, $G'(f) = G(f)$ if $f$ belongs to $\mcalb^{(r)}(M)$. Otherwise, $G'(f)$ is the zero morphism. 

\begin{defn} \label{defn:hl1}
	Let $F$ be an object of $\mcall(\mcalo(\spkm); \mcalm)$. Consider the restriction cofunctor $res(\Lb_r(F)) \colon \Bb^{(r)}(\spkm) \to \mcalm$. Also consider the cofunctor $\Psi (res (\Lb_r(F))) \colon \mcalb^{(r)}(M) \to \mcalm$ given by 
	\[
	\Psi (res (\Lb_r(F))) (U) = res(\Lb_r(F)) (\theta (U)) = \Lb_r(F)(\theta (U)).
	\]
	The image of this latter cofunctor under $\psi$ gives a homogeneous cofunctor $\psi \left(\Psi (res (\Lb_r(F))) \right) \colon \om \to \mcalm$ of degree $r$. Though $\psi$ is not a genuine Kan extension, we will write
	\begin{equation} \label{eqn:hl3}
	\Psi (res (\Lb_r(F)))^!:= \psi \left(\Psi (res (\Lb_r(F))) \right). 
	\end{equation}
\end{defn}

\begin{thm} \label{thm:hl}
	Suppose $\mcalm$ is a simplicial model category that has a zero object (see Remarkk~\ref{rmk:hl} below).  Let $F \in \mcall(\mcalo(\spkm); \mcalm)$. Then the cofunctors $L_r(\Psi (res (F))^!)$ and $\Psi (res(\Lb_r(F)))^!$ from (\ref{eqn:hl01}) and (\ref{eqn:hl3}) respectively  are connected by a zigzag of weak equivalences in the category of homogeneous cofunctors $\om \to \mcalm$ of degree $r$. That is, 
	\[
	L_r(\Psi (res (F))^!) \simeq \Psi (res(\Lb_r(F)))^! \quad \text{or just} \quad L_r(\Psi (F)^!) \simeq \Psi (\Lb_r(F))^!. 
	\]
\end{thm}

\begin{proof}
	Since the cofunctors $L_r(\Psi(F)^!)$ and $\Psi(\Lb_r(F))^!$ are both homogeneous of degree $r$, by \cite[Lemma 6.5]{paul_don17-2}, it is enough to show that they  agree (up to equivalence) on $\mcalb^{(r)}(M)$. Let $U \in \mcalb^{(r)}(M)$. The idea is to construct a zigzag $L_r(\Psi(F)^!) (U) \stackrel{\sim}{\lla} \cdots \stackrel{\sim}{\lra} \Psi(\Lb_r(F))^!(U)$ of weak equivalences natural in $U$.  Consider the following diagram. 
	\[
	\xymatrix{L_r(\Psi(F)^!)(U)   &   \hfib\left( T_r (\Psi(F)^!)(U) \to T_{r-1}(\Psi(F)^!)(U) \right) \ar[l]_-{=} \\
		\hfib\left( \Psi(F)^! (U) \to T_{r-1}(\Psi(F)^!)(U)\right) \ar[ru]^-{\sim} &  \hfib\left( \Psi(F)(U) \to T_{r-1}(\Psi(F)^!)(U)\right) \ar[l]_-{\sim} \\
		\hfib\left( F(\theta(U)) \to T_{r-1}(\Psi(F)^!)(U)\right) \ar[ru]^-{=} &  \hfib\left( F(\theta(U)) \to \underset{W \in \mcalb_{r-1}(U)}{\text{holim}} \Psi(F)^!(W)\right) \ar[l]_-{=} \\
		\hfib\left( F(\theta(U)) \to \underset{W \in \mcalb_{r-1}(U)}{\text{holim}} \Psi(F)(W)\right) \ar[ru]^-{\sim} &  \hfib\left( F(\theta(U)) \to \underset{W \in \mcalb_{r-1}(U)}{\text{holim}} F(\theta(W))\right) \ar[l]_-{=}}
	\]
	Here \lq\lq $=$\rq\rq{} stands for the identity morphism. The second map is induced by the canonical map $\Psi(F)^!(U) \to T_r (\Psi(F)^!)(U) = \underset{W \in \mcalb_r(U)}{\text{holim}}  \Psi(F)^!(W)$, which comes from the fact that $U$ is the terminal object of $\mcalb_r(U)$. The same fact implies that this latter map is a weak equivalence. Using the same argument, we obtain the third and sixth maps. It is straightforward to check that all these maps are natural in $U$. 
	
	
	On the other hand, consider the following diagram. 
	\[
	\xymatrix{ \Psi(\Lb_r(F))^!(U)  & \Psi(\Lb_r(F))(U) \ar[l]_-{\sim} \\  \Lb_r(F)(\theta(U)) \ar[ru]^-{=} & L'_r(F)(\theta(U)) \ar[l]_-{\sim} \\
		\hfib \left( F (\theta(U)) \to F(\theta(U) \backslash \widehat{\theta(U)})\right) \ar[ru]^-{=} \ar[r]^-{\sim} &   \hfib \left( F (\theta(U)) \to res(F)^!(\theta(U) \backslash \widehat{\theta(U)})\right) \ar[d]^-{=} \\
		   &
		\hfib \left( F (\theta(U)) \to \underset{W \in \Bb(\theta(U) \backslash \widehat{\theta(U)})}{\text{holim}} F(W)\right) }
	\]
	The first map comes from the part of Lemma~\ref{lem:hl-phipsi} saying that the canonical map $id \to \phi \psi$ is a weak equivalence. The third comes from Lemma~\ref{lem:hl-Lbr}. The  fifth map is induced by the map $F \to res(F)^!$, where $res$ and $(-)^!$ are the functors from (\ref{eqn:hl0}).  This latter map is a weak equivalence by Lemma~\ref{lem:mcall_mcalfab}. Again, it is straightforward to check that all these maps are natural in $U$. 
	
	
	Now consider the following commutative triangle.
	\[
	\xymatrix{\mcalb_{r-1} (U) \ar[rr]^-{F \theta} \ar[d]_-{\theta}  &   &  \mcalm  \\
		\Bb(\theta(U) \backslash \widehat{\theta(U)}) \ar[rru]_-{F} &  &  }
	\]
	This  induces a  map 
	$
	\underset{W \in \Bb(\theta(U) \backslash \widehat{\theta(U)})}{\text{holim}} F(W) \to \underset{W \in \mcalb_{r-1}(U)}{\text{holim}} F(\theta(W)),
	$
	which is a weak equivalence since $\theta$ is homotopy right cofinal. (Indeed,  for every $V \in \Bb(\theta(U) \backslash \widehat{\theta(U)})$, the under category $V \downarrow \theta$  has a terminal object, namely $V'$, where $V'$ is the union of the components of $U$ containing $\lambda(V)$.) One can see that this latter map induces a weak equivalence
	\[
	\hfib \left( F (\theta(U)) \to \underset{W \in \Bb(\theta(U) \backslash \widehat{\theta(U)})}{\text{holim}} F(W)\right)  \stackrel{\sim}{\lra}  \hfib\left( F(\theta(U)) \to \underset{W \in \mcalb_{r-1}(U)}{\text{holim}} F(\theta(W))\right),
	\]
	which is also natural in $U$. This proves the theorem. 
\end{proof}

\begin{rmk} \label{rmk:hl}
	By inspection, one can notice that in Theorem~\ref{thm:hl} the hypothesis that $\mcalm$ has a zero object is needed for all $2 \leq r \leq k$, but not for $r =1$.  
\end{rmk}


\textsf{University of Regina, 3737 Wascana Pkwy, Regina, SK S4S 0A2, Canada\\
Department of Mathematics and Statistics\\}
\textit{E-mail address: pso748@uregina.ca}

\textsf{University of Regina, 3737 Wascana Pkwy, Regina, SK S4S 0A2, Canada\\
Department of Mathematics and Statistics\\}
\textit{E-mail address: donald.stanley@uregina.ca}

\end{document}